\documentclass[12pt, a4paper]{amsart}
\usepackage{graphicx}
\usepackage{setspace}
\usepackage{tikz}
\usetikzlibrary{patterns}
\usepackage{csquotes}
\usepackage{float}          
\usepackage[dvipsnames]{xcolor} 
\usepackage{tikz}
\usetikzlibrary{calc}       
\usepackage{tkz-euclide}    
\definecolor{darkred}{RGB}{145,0,0}
\definecolor{darkgreen}{RGB}{0,120,0}

\usepackage{enumitem}
\usepackage{amssymb}
\usepackage{listings}
\usepackage{xcolor,etoolbox}
\usepackage[
    colorlinks=true,
    linkcolor=darkred,
    urlcolor=darkgreen,
    citecolor=blue
]{hyperref}
\usepackage{doi}
\usepackage{orcidlink}

\theoremstyle{definition}
\newtheorem{theorem}{Theorem}[section]
\numberwithin{theorem}{section}
\newtheorem{definition}[theorem]{Definition}

\newtheorem{proposition}[theorem]{Proposition}

\title[A Modular Form Proof of the Irrationality of $\zeta\left(3\right)$]
 {A Modular Form Proof of the Irrationality of $\zeta\left(3\right)$} 
 \author{Pang Ern Thang \orcidlink{0009-0005-3447-4188}}
\address{Department of Mathematics, National University of Singapore, 10 Lower Kent Ridge Road, Singapore 119076}
\email{e1398618@u.nus.edu}

\begin{document}
\maketitle \begin{abstract}
We present an expository proof of the irrationality of $\zeta\left(3\right)$ using modular forms of level 6. By constructing a suitable Eichler integral, we obtain a power series with controlled denominators and sufficiently large radius of convergence. Beukers' irrationality criterion then implies that $\zeta\left(3\right)$ is irrational.
\end{abstract}
\section{Introduction}
Let \[\zeta\left(s\right)=\sum_{n=1}^{\infty}\frac{1}{n^s}\]
denote the Riemann zeta function and $\mathbb{N}=\left\{1,2,\ldots\right\}$ denote the set of natural numbers. For any positive real number $s>1$, $\zeta\left(s\right)$ can provide the sum of various convergent infinite series, such as $\zeta\left(2\right)$. In 1740, Euler evaluated the zeta function at even positive integers and deduced that for any $n\in\mathbb{N}$, \[\zeta\left(2n\right)=\left(-1\right)^{n+1}\frac{B_{2n}\left(2\pi\right)^{2n}}{2\left(2n\right)!},\]
where $B_n$ denotes the sequence of Bernoulli numbers \cite{Apostol1973}. It is known that $B_n\in\mathbb{Q}$ for all $n\in\mathbb{N}$ so $\zeta\left(2n\right)$ is indeed irrational. As for the odd positive integers, no such simple expression is known. Having said that, it is of interest whether $\zeta\left(2n+1\right)$ is irrational, though it is conjectured that it is for all $n\in\mathbb{N}$.

In 1978, Apéry showed that $\zeta\left(3\right)$ is irrational \cite{Beukers1979}. This constant is also known as Apéry's constant. Many people have tried to extend Apéry's proof that $\zeta\left(3\right)$ is irrational to other values of $\zeta\left(s\right)$ with odd arguments. Even though this has so far not produced any results on other numbers, in 2001, Ball and Rivoal showed that infinitely many of the $\zeta\left(2n+1\right)$ are irrational \cite{BallRivoal2001}, and in 2002, Zudilin showed that at least one of $\zeta\left(5\right)$, $\zeta\left(7\right)$, $\zeta\left(9\right)$, $\zeta\left(11\right)$ is irrational \cite{Zudilin2001}. 

In this paper, we will see that Apéry's proof of $\zeta\left(3\right)$ is a nice consequence of complex analysis on spaces of certain modular forms. Having said that, it is difficult to determine the irrationality of other odd values of $\zeta\left(s\right)$ using modular forms.

We first state some preliminaries before delving into the main parts of the proof. First, define \[t\left(q\right)=\sum_{n=0}^{\infty}t_nq^n\] to be a convergent power series in $q$ for all $\left|q\right|<1$. Let $w\left(q\right)$ be another analytic function on $\left|q\right|<1$. We like to study $w$ as a function of $t$. In general, it will be a multivalued function over which we have no control. However, we can make some assumptions. Suppose $t_0=0$ and $t_1\ne 0$. Let $q\left(t\right)$ denote the local inverse of $t\left(q\right)$ with $q\left(0\right)=0$. Let $w\left(q\left(t\right)\right)$ denote the value of $w$ around $t=0$. To determine the radius of convergence of the power series \[w\left(q\left(t\right)\right)=\sum_{n=0}^{\infty}w_nt^n,\]
we introduce \emph{branching values} of $t$.
\begin{definition}[branching value] Let
\[
t:\mathbb{D}\rightarrow\mathbb{C}
\quad\text{where}\quad \mathbb{D}=\left\{q\in\mathbb{C}:\left|q\right|<1\right\}
\]
be a non-constant holomorphic function. We say that \(t\) \emph{branches above} \(t_0\in\mathbb{C}\) if either
\begin{itemize}
\itemsep 0em
    \item \(t_0\notin t\left(\mathbb{D}\right)\), or
    \item there exists \(q_0\in\mathbb{D}\) such that $t\left(q_0\right)=t_0$ and $t'\left(q_0\right)=0$.
\end{itemize}
Equivalently, \(t\) branches above \(t_0\) if \(t\) is not a local covering above \(t_0\). Such a point \(t_0\) is called a \emph{branching value} of \(t\).
\end{definition}
Assume that the branching values of \(t\) form a discrete set $t_1,t_2,\ldots$, with \(0\) excluded, and order them so that $\left|t_1\right|<\left|t_2\right|<\cdots$. In general, the nearest branching value \(t_1\) determines the radius of convergence of \(w\left(q\left(t\right)\right)\), which is therefore typically \(\left|t_1\right|\). We are interested in situations where the function extends beyond this first obstruction.

Let \(\gamma\) be a closed contour in the complex \(t\)-plane based at the origin. Suppose that \(\gamma\) avoids all branching values and winds exactly once around \(t_1\). If analytic continuation of \(w\left(q\left(t\right)\right)\) along \(\gamma\) returns to the original branch, then the function admits analytic continuation throughout $\left|t\right|<\left|t_2\right|$ apart from a possible isolated singularity at \(t_1\). If, moreover, \(w\left(q\left(t\right)\right)\) remains bounded in a punctured neighbourhood of \(t_1\), then this singularity is removable. Consequently, the Taylor series of \(w\left(q\left(t\right)\right)\) about \(t=0\) has radius of convergence at least \(\left|t_2\right|\).

The irrationality arguments considered below are based on constructing examples for which this enlargement of the radius of convergence occurs. Proposition \ref{prop: Beukers' irrationality criterion} provides a systematic method for obtaining a radius of convergence that is as large as possible.
\begin{proposition}[Beukers' irrationality criterion]\label{prop: Beukers' irrationality criterion}
Let $f_0\left(t\right),f_1\left(t\right),\ldots,f_k\left(t\right)$ be power series in $t$. Suppose for any $n\in\mathbb{N}$ and $i=0,1,\ldots,k$, the $n^\text{th}$ coefficient in the Taylor series of $f_i$ is rational and has denominator dividing $d^n\left(\operatorname{lcm}\left(1,\ldots,n\right)\right)^r$, where $r,d\in\mathbb{N}$ are fixed. Suppose there exist $\theta_1,\ldots,\theta_k\in\mathbb{R}$ such that $f_0\left(t\right)+\theta_1f_1\left(t\right)+\cdots+\theta_kf_k\left(t\right)$ has radius of convergence $\rho$ and infinitely many non-zero Taylor coefficients. If $\rho>de^r$, then at least one of $\theta_1,\ldots,\theta_k$ is irrational.
\end{proposition}
\begin{proof}
Choose $\varepsilon>0$ such that $\rho-\varepsilon>de^{r\left(1+\varepsilon\right)}$. Let \[f_i\left(t\right)=\sum_{n=0}^{\infty}a_{in}t^n.\]
Since the radius of convergence of $f_0\left(t\right)+\theta_1f_1\left(t\right)+\cdots+\theta_kf_k\left(t\right)$ is $\rho$, then by the Cauchy-Hadamard formula,
\[\frac{1}{\rho}=\limsup_{n\to\infty}\left|c_n\right|^{1/n}\quad\text{where }c_n=a_{0n}+a_{1n}\theta_1+\cdots+a_{kn}\theta_k.\]
So, for sufficiently large $n$, \[\left|a_{0n}+a_{1n}\theta_1+\cdots+a_{kn}\theta_k\right|<\frac{1}{\left(\rho-\varepsilon\right)^n}.\]
Suppose $\theta_1,\ldots,\theta_k\in\mathbb{Q}$ and have common denominator $D$. Then, \[\mathbb{Z}\ni A_n=Dd^n\operatorname{lcm}\left(1,\ldots,n\right)^r\left|a_{0n}+a_{1n}\theta_1+\cdots+a_{kn}\theta_k\right|<\frac{Dd^n\left(\operatorname{lcm}\left(1,\ldots,n\right)\right)^r}{\left(\rho-\varepsilon\right)^n}.\]
Note that $\operatorname{lcm}\left(1,\ldots,n\right)=e^{\psi\left(n\right)}$, where $\psi$ denotes the Chebyshev psi function. By the prime number theorem, $\psi\left(n\right)\sim n$ so $\operatorname{lcm}\left(1,\ldots,n\right)=e^{n+o\left(n\right)}$ \cite{Apostol1976}. Hence, $\operatorname{lcm}\left(1,\ldots,n\right)<e^{\left(1+\varepsilon\right)n}$ for sufficiently large $n$. This implies \[\left|A_n\right|<D\left(\frac{de^{\left(1+\varepsilon\right)r}}{\rho-\varepsilon}\right)^n.\]
Since \[\frac{de^{\left(1+\varepsilon\right)r}}{\rho-\varepsilon}<1,\]
then $A_n=0$ for sufficiently large $n$. However, this contradicts our assumption that $A_n\ne 0$ for infinitely many $n$.
\end{proof}
Next, we would need to construct $t\left(q\right)$ and $w\left(q\right)$. We will do so via modular forms. The values for which we obtain irrationality results are in fact values at integral points of Dirichlet series associated to modular forms.
\begin{proposition}\label{prop: h formula in terms of d}
Let $q=e^{2\pi i\tau}$ and \[F\left(\tau\right)=\sum_{n=1}^{\infty}a_nq^n\] be a Fourier series convergent for $\left|q\right|<1$ such that for some $k,N\in\mathbb{N}$ and $\varepsilon\in\left\{-1,1\right\}$, \[F\left(-\frac{1}{N\tau}\right)=\varepsilon\left(-i\tau\sqrt{N}\right)^kF\left(\tau\right).\] Let \[f\left(\tau\right)=\sum_{n=1}^{\infty}\frac{a_n}{n^{k-1}}q^n.\] Let \[L\left(F,s\right)=\sum_{n=1}^{\infty}\frac{a_n}{n^s}\] be the associated Dirichlet series and define \[h\left(\tau\right)=f\left(\tau\right)-\sum_{0\le r<\frac{k-2}{2}}\frac{L\left(F,k-r-1\right)}{r!}\left(2\pi i\tau\right)^r.\] Then, \[h\left(\tau\right)-D=\left(-1\right)^{k-1}\varepsilon\left(-i\tau\sqrt{N}\right)^{k-2}h\left(-\frac{1}{N\tau}\right),\] where \[D=\begin{cases}0&\text{if $k$ is odd};\\[4pt]\displaystyle\frac{L\left(F,\frac{k}{2}\right)}{\left(\frac{k}{2}-1\right)!}\left(2\pi i\tau\right)^{\frac{k}{2}-1}&\text{if $k$ is even}.\end{cases}\]
\end{proposition}
\begin{proof}
 Assume that $k\ge 2$. Since $F$ has no constant Fourier coefficient, $F\left(iy\right)$ decays exponentially as $y\to\infty$. Moreover, the Fricke involution eigenvalue relation \cite{Miyake2006} yields
 \[F\left(-\frac{1}{N\tau}\right)=\varepsilon\left(-i\tau\sqrt{N}\right)^kF\left(\tau\right).\]
 Let $\tau=iy$, so
 \[F\left(iy\right)=\varepsilon\left(y\sqrt{N}\right)^{-k}F\left(\frac{i}{Ny}\right).\]
 As such, $F\left(iy\right)$ decays exponentially as $y\to0^+$. Also, on the imaginary axis, \[F\left(iy\right)=\sum_{n=1}^{\infty}a_ne^{-2\pi n y}.\]
 Again, as there is no constant Fourier coefficient, $F\left(iy\right)$ decays exponentially as $y\to\infty$. Hence, all the integrals below converge. We claim that \begin{align}\label{eqn: claim eichler integral}
    f\left(\tau\right)=\frac{\left(-2\pi i\right)^{k-1}}{\left(k-2\right)!}\int_{\tau}^{i\infty}\left(z-\tau\right)^{k-2}F\left(z\right)\;dz.
 \end{align} 
By the Fourier expansion \[F\left(z\right)=\sum_{n=1}^{\infty}a_ne^{2n\pi iz},\]
it suffices to compute \[\int_{\tau}^{i\infty}\left(z-\tau\right)^{k-2}e^{2n\pi iz}\;dz.\]
We parametrise using $z=\tau + it$, where $t\ge 0$. Then, \begin{align*}
    \int_{\tau}^{i\infty}\left(z-\tau\right)^{k-2}e^{2n\pi iz}\;dz=e^{2n\pi i\tau}\int_{0}^{\infty}\left(it\right)^{k-2}e^{-2n\pi t}i\;dt=i^{k-1}e^{2n\pi i\tau}\int_{0}^{\infty}t^{k-2}e^{-2n\pi t}dt.
\end{align*}
Using integration by parts, this is equal to \[\frac{i^{k-1}\left(k-2\right)!}{\left(2\pi n\right)^{k-1}}e^{2\pi in\tau}.\]
Since \[f\left(\tau\right)=\sum_{n=1}^{\infty}\frac{a_n}{n^{k-1}}q^n=\sum_{n=1}^{\infty}\frac{a_n}{n^{k-1}}e^{2n\pi i\tau},\]
then indeed, our formula for $f\left(\tau\right)$ in (\ref{eqn: claim eichler integral}) holds because \[f\left(\tau\right)=\sum_{n=1}^{\infty}\frac{\left(-2\pi i\right)^{k-1}}{\left(k-2\right)!}\frac{i^{k-1}\left(k-2\right)!}{\left(2\pi n\right)^{k-1}}a_ne^{2\pi in\tau}=\sum_{n=1}^{\infty}\frac{a_n}{n^{k-1}}e^{2n\pi i\tau}.\]
Let
\[W_N\tau=-\frac{1}{N\tau}\quad\text{and}\quad A\left(\tau\right)=\left(-1\right)^{k-1}\varepsilon\left(-i\tau\sqrt{N}\right)^{k-2}.\] 
Then, \[A\left(\tau\right)f\left(W_N\tau\right)=\frac{\left(-2\pi i\right)^{k-1}}{\left(k-2\right)!}\int_{\tau}^{0}\left(z-\tau\right)^{k-2}F\left(z\right)\;dz.\] As such, \[f\left(\tau\right)-A\left(\tau\right)f\left(W_N\tau\right)=\frac{\left(-2\pi i\right)^{k-1}}{\left(k-2\right)!}\int_{0}^{i\infty}\left(z-\tau\right)^{k-2}F\left(z\right)\;dz.\] Expanding $\left(z-\tau\right)^{k-2}$ using the binomial theorem, \[\left(z-\tau\right)^{k-2}=\sum_{r=0}^{k-2}\binom{k-2}{r}z^{k-r-2}\left(-\tau\right)^r.\]
Also, using the substitution $z=iy$, \begin{align*}
\int_{0}^{i\infty}z^mF\left(z\right)\;dz&=i^{m+1}\int_{0}^{\infty}y^mF\left(iy\right)\;dy\\\
&=i^{m+1}\sum_{n=1}^{\infty}a_n\int_{0}^{\infty}y^me^{-2n\pi y}\;dy\\
&=\frac{i^{m+1}m!}{\left(2\pi\right)^{m+1}}L\left(F,m+1\right)
\end{align*}
Hence,
\begin{align}\label{eqn: f(tau) - A(tau) etc.}
f\left(\tau\right)-A\left(\tau\right)f\left(W_N\tau\right)=\sum_{r=0}^{k-2}\frac{L\left(F,k-r-1\right)}{r!}\left(2\pi i\tau\right)^r.
\end{align} We now use the functional equation of $L\left(F,s\right)$. Define the completed $L$-function by \[\Lambda\left(F,s\right)=N^{s/2}\left(2\pi\right)^{-s}\Gamma\left(s\right)L\left(F,s\right).\] By the Mellin transform, \[\left(2\pi\right)^{-s}\Gamma\left(s\right)L\left(F,s\right)=\int_{0}^{\infty}F\left(iy\right)y^{s-1}\;dy\] together with the substitution $y\mapsto1/\left(Ny\right)$, we obtain the functional equation
\begin{align}\label{eqn: L function functional equation}
\Lambda\left(F,s\right)=\varepsilon\Lambda\left(F,k-s\right).
\end{align} Taking $s=k-r-1$ in (\ref{eqn: L function functional equation}) yields \begin{align}\label{eqn: l function relation}
\frac{L\left(F,k-r-1\right)}{r!}=\varepsilon N^{r+1-k/2}\left(2\pi\right)^{k-2r-2}\frac{L\left(F,r+1\right)}{\left(k-r-2\right)!}
\end{align} Put \[Q\left(\tau\right)=\sum_{0\le r<\frac{k-2}{2}}\frac{L\left(F,k-r-1\right)}{r!}\left(2\pi i\tau\right)^r,\] so that $h\left(\tau\right)=f\left(\tau\right)-Q\left(\tau\right)$. Applying (\ref{eqn: l function relation}) and replacing $r$ by $k-r-2$, we obtain \begin{align}\label{eqn: sum of l function} 
\sum_{\frac{k-2}{2}<r\le k-2}\frac{L\left(F,k-r-1\right)}{r!}\left(2\pi i\tau\right)^r=-A\left(\tau\right)Q\left(W_N\tau\right).
\end{align} If $k$ is odd, the polynomial on the right side of (\ref{eqn: f(tau) - A(tau) etc.}) has no middle term, and we set $D=0$. If $k$ is even, its middle term corresponds to $r=\frac{k}{2}-1$ and is \[D=\frac{L\left(F,\frac{k}{2}\right)}{\left(\frac{k}{2}-1\right)!}\left(2\pi i\tau\right)^{\frac{k}{2}-1}.\] Thus, by (\ref{eqn: f(tau) - A(tau) etc.}) and (\ref{eqn: sum of l function}), \[f\left(\tau\right)-A\left(\tau\right)f\left(W_N\tau\right)=Q\left(\tau\right)+D-A\left(\tau\right)Q\left(W_N\tau\right).\] Rearranging gives \[h\left(\tau\right)-D=A\left(\tau\right)h\left(W_N\tau\right)=\left(-1\right)^{k-1}\varepsilon\left(-i\tau\sqrt{N}\right)^{k-2}h\left(-\frac{1}{N\tau}\right).\] Finally, setting $s=\frac{k}{2}$ in (\ref{eqn: L function functional equation}) gives \[\Lambda\left(F,\frac{k}{2}\right)=\varepsilon\Lambda\left(F,\frac{k}{2}\right).\] Consequently, if $\varepsilon=-1$, then $\Lambda\left(F,\frac{k}{2}\right)=0$, and hence \[L\left(F,\frac{k}{2}\right)=0.\qedhere\]
\end{proof}
\section{Modular Forms and the Group $\Gamma_1\left(6\right)$}
\begin{definition}[modular form]\label{definition of modular form} Let $\Gamma<\operatorname{SL}_2\left(\mathbb{Z}\right)$ be a subgroup of finite index. Then, a modular form of level $\Gamma$ and weight $k$ is a holomorphic function $f:\mathcal{H}\to\mathbb{C}$, where $\mathcal{H}$ denotes the upper half-plane, satisfying the following conditions:
\begin{itemize}
    \itemsep 0em
    \item Automorphy condition: for any $\gamma\in \Gamma$, $f\left(\gamma\left(z\right)\right)=\left(cz+d\right)^kf\left(z\right)$
    \item Growth condition: for any $\gamma\in\operatorname{SL}_2\left(\mathbb{Z}\right)$, $\left(cz+d\right)^{-k}f\left(\gamma\left(z\right)\right)$ is bounded as $\operatorname{Im}\left(z\right)\to\infty$
\end{itemize}
\end{definition}
In Definition \ref{definition of modular form}, \[\gamma=\begin{bmatrix}
    a& b\\c&d
\end{bmatrix}\in\operatorname{SL}_2\left(\mathbb{Z}\right)\text{ is a matrix}\quad\text{and}\quad\text{it is identified with the function }\gamma\left(z\right)=\frac{az+b}{cz+d}.\]
The identification of functions with matrices makes function composition equivalent to matrix multiplication.
\begin{definition}[modular function] Let $\Gamma\subseteq\operatorname{SL}_2\left(\mathbb{Z}\right)$ be a congruence subgroup acting on $\mathcal{H}$ by Möbius transformations \[\frac{a\tau+b}{c\tau+d}\quad\text{corresponding to }\gamma=\begin{bmatrix}
    a&b\\c&d
\end{bmatrix}\in \Gamma.\]
Then, a modular function for $\Gamma$ is a function $f:\mathcal{H}\to\mathbb{C}$ satisfying the following properties:
\begin{itemize}
    \itemsep 0em
    \item $f$ is meromorphic on $\mathcal{H}$
    \item $f$ is invariant under $\Gamma$. That is to say, \[f\left(\frac{a\tau+b}{c\tau+d}\right)=f\left(\tau\right)\quad\text{for all }\gamma\in \Gamma.\]
    \item $f$ is meromorphic at every cusp of $\Gamma$. In other words, at each cusp, $f$ has a Fourier expansion with only finitely many negative-power terms.
\end{itemize}
\end{definition}

As in \cite{DiamondShurman2005}, let $\Gamma_1\left(6\right)$ denote the \emph{congruence subgroup} of level 6 defined by \[\Gamma_1\left(6\right)=\left\{\begin{bmatrix}
    a&b\\c&d
\end{bmatrix}\in\operatorname{SL}_2\left(\mathbb{Z}\right):a\equiv d\equiv 1\text{ }\left(\operatorname{mod}6\right)\text{ and }c\equiv 0 \text{ }\left(\operatorname{mod}6\right)\right\}.\]
Analogously, define the \emph{congruence subgroup} of level 6 $\Gamma_0\left(6\right)$ by 
\[\Gamma_0\left(6\right)=\left\{\begin{bmatrix}
    a&b\\c&d
\end{bmatrix}\in\operatorname{SL}_2\left(\mathbb{Z}\right):c\equiv 0 \text{ }\left(\operatorname{mod}6\right)\right\}.\]
To visualise $\Gamma_1\left(6\right)$, please refer to \cite{LowryDuda2021}. Also, see Figures \ref{fig:gamma 1 (6) visualisation} and \ref{fig:domain colouring of gamma 1 (6)} for some visualisations of $\Gamma_1\left(6\right)$ and Figure \ref{fig:domain colouring of gamma 0 (6)} for a visualisation of $\Gamma_0\left(6\right)$.
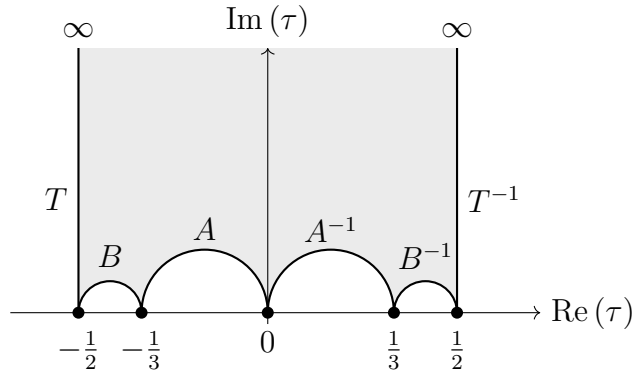
\begin{figure}[H]
    \centering
    \begin{tikzpicture}[scale=5]

\def\H{0.7}

\fill[black!8]
    (-1/2,\H)
    -- (-1/2,0)
    arc[
        start angle=180,
        end angle=0,
        radius=1/12
    ]
    arc[
        start angle=180,
        end angle=0,
        radius=1/6
    ]
    arc[
        start angle=180,
        end angle=0,
        radius=1/6
    ]
    arc[
        start angle=180,
        end angle=0,
        radius=1/12
    ]
    -- (1/2,\H)
    -- cycle;

\draw[->]
    (-0.68,0)
    --
    (0.72,0)
    node[right]
    {$\operatorname{Re}\left(\tau\right)$};

\draw[->]
    (0,-0.03)
    --
    (0,0.7)
    node[above]
    {$\operatorname{Im}\left(\tau\right)$};

\draw[thick]
    (-1/2,0)
    --
    (-1/2,\H);

\draw[thick]
    (1/2,0)
    --
    (1/2,\H);

\draw[thick]
    (-1/2,0)
    arc[
        start angle=180,
        end angle=0,
        radius=1/12
    ]
    arc[
        start angle=180,
        end angle=0,
        radius=1/6
    ]
    arc[
        start angle=180,
        end angle=0,
        radius=1/6
    ]
    arc[
        start angle=180,
        end angle=0,
        radius=1/12
    ];

\foreach \x/\lab in {
    -0.5/{-\frac{1}{2}},
    -0.333333/{-\frac{1}{3}},
     0/{0},
     0.333333/{\frac{1}{3}},
     0.5/{\frac{1}{2}}
}{
    \fill (\x,0) circle (0.45pt);
    \node[below=3pt] at (\x,0) {$\lab$};
}

\node[above] at (-1/2,\H) {$\infty$};
\node[above] at (1/2,\H) {$\infty$};

\node[left] at (-1/2,0.3) {$T$};
\node[right] at (1/2,0.3) {$T^{-1}$};

\node at (-1/6,0.22) {$A$};
\node at (1/6,0.22) {$A^{-1}$};

\node at (-5/12,0.15) {$B$};
\node at (5/12,0.15) {$B^{-1}$};

\end{tikzpicture}
    \caption{Visualising a fundamental domain for the action of $\Gamma_1\left(6\right)$ on $\mathcal{H}$}
    \label{fig:gamma 1 (6) visualisation}
\end{figure}

\begin{figure}[H]
    \centering
    \includegraphics[width=0.8\linewidth]{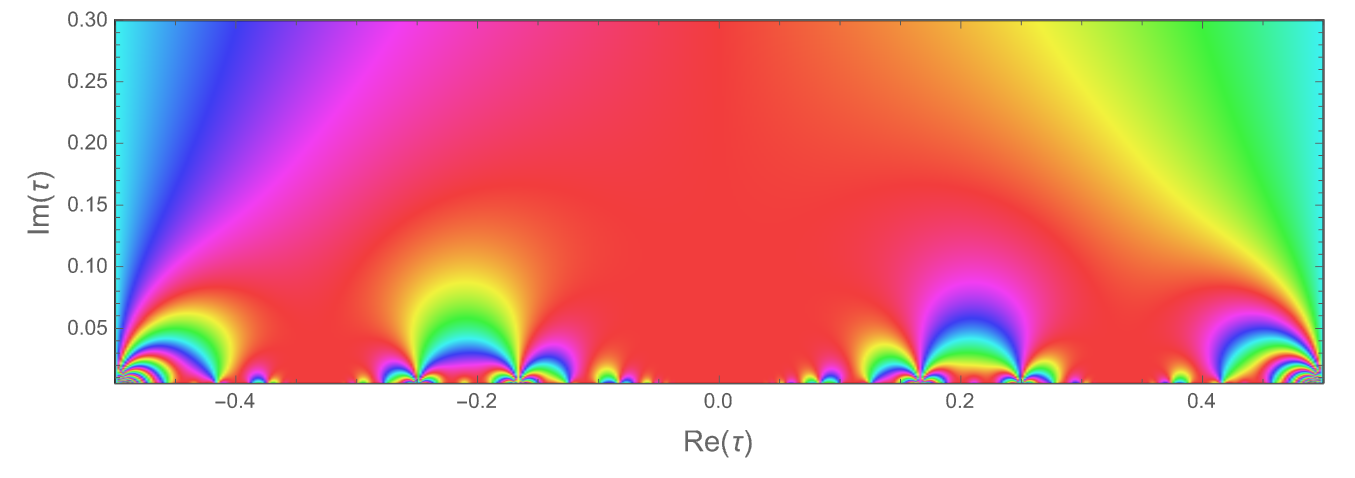}
    \caption{Domain colouring of $\Gamma_1\left(6\right)$}
    \label{fig:domain colouring of gamma 1 (6)}
\end{figure}

\begin{figure}[H]
    \centering
    \includegraphics[width=0.8\linewidth]{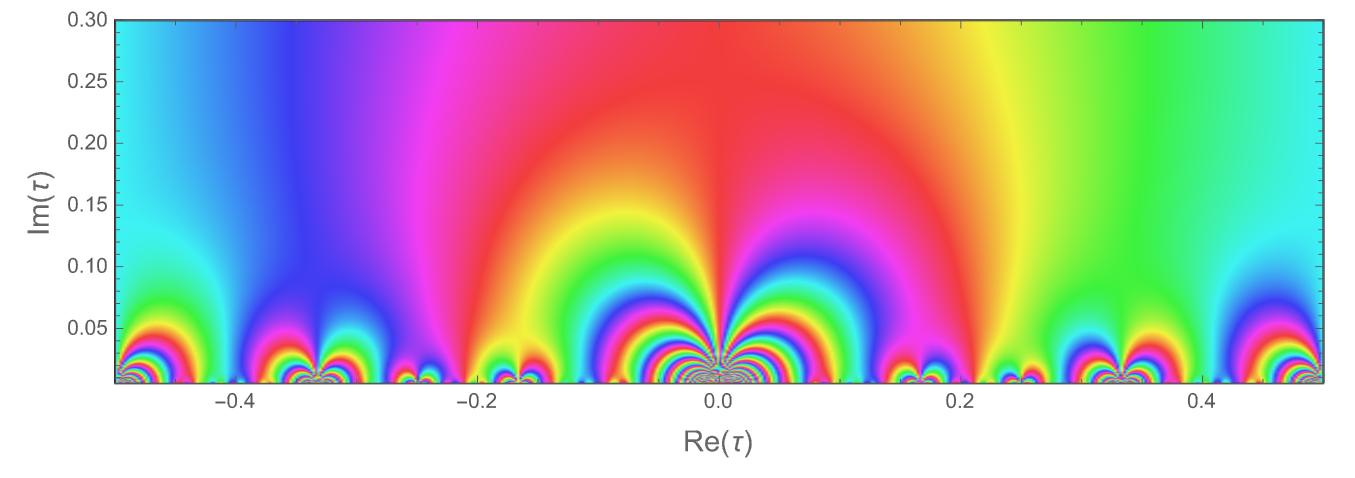}
    \caption{Domain colouring of $\Gamma_0\left(6\right)$}
    \label{fig:domain colouring of gamma 0 (6)}
\end{figure}

Again, let $q=e^{2\pi i\tau}$, where $\operatorname{Im}\left(\tau\right)>0$. Define \begin{align}\label{eqn: definition of y}
y\left(\tau\right)=\frac{\left(\eta\left(6\tau\right)\right)^8\left(\eta\left(\tau\right)\right)^4}{\left(\eta\left(2\tau\right)\right)^8\left(\eta\left(3\tau\right)\right)^4}\quad\text{where }\eta\left(\tau\right)=q^{1/24}\prod_{n=1}^{\infty}\left(1-q^n\right).
\end{align}
Note that $\eta$, known as the Dedekind eta function, is a modular form of weight $1/2$ and it is a function defined on $\mathcal{H}$, where $\operatorname{Im}\left(\tau\right)>0$.
\begin{proposition}\label{prop: y is modular gamma 1 6}
The function \(y\) is a modular function on \(\Gamma_1\left(6\right)\).
\end{proposition}
\begin{proof}
Since each $\eta$ is holomorphic and non-zero on $\mathcal{H}$, then in the formula for $y$ (\ref{eqn: definition of y}), the denominator never vanishes. So, $y$ is holomorphic on $\mathcal{H}$. 

Then, we prove that $y$ is invariant for every $\gamma\in \Gamma_0\left(6\right)$. We write \[y\left(\tau\right)=\prod_{\delta\mid 6}\eta\left(\delta \tau\right)^{r_\delta},\]
where $r_1=4$, $r_2=-8$, $r_3=-4$, and $r_6=8$. The weight of this eta quotient is \[k=\frac{1}{2}\sum_{\delta\mid 6}r_\delta=\frac{4-8-4+8}{2}=0.\]
By Newman's eta quotient criterion \cite{RouseWebb2015},
\begin{align*}
\sum_{\delta\mid 6}\delta r_\delta=1\left(4\right)+2\left(-8\right)+3\left(-4\right)+6\left(8\right)=24\equiv 0\pmod{24}
\end{align*}
and
\begin{align*}
\sum_{\delta\mid 6}\frac{6}{\delta}r_\delta=6\left(4\right)+3\left(-8\right)+2\left(-4\right)+1\left(8\right)=0\equiv 0\pmod{24}.
\end{align*}
Thus, the two congruence conditions in the eta-quotient criterion are satisfied. The associated character is
\[\chi\left(d\right)=\left(\frac{(-1)^k\prod_{\delta\mid 6}\delta^{r_\delta}}{d}\right),
\]
where the expression on the right is the Kronecker symbol. We have
\begin{align*}
\prod_{\delta\mid 6}\delta^{r_\delta}=81.
\end{align*}
Since \(81=9^2\) is a square, we have
\[\chi\left(d\right)=\left(\frac{81}{d}\right)=1\]
whenever \(\gcd\left(d,6\right)=1\). Hence, for every
\[\gamma=
\begin{bmatrix}
a&b\\
c&d
\end{bmatrix}\in\Gamma_0\left(6\right),
\]
the eta quotient transformation law gives
\[y\left(\gamma\tau\right)=\chi\left(d\right)\left(c\tau+d\right)^k y\left(\tau\right)=y\left(\tau\right)\]
because \(k=0\) and \(\chi\left(d\right)=1\). This proves the required invariance.

Lastly, we prove that $f$ is meromorphic at every cusp of $\Gamma$. Based on our earlier discussion, $y$ is holomorphic and non-zero on $\mathcal{H}$. As such, any zeros or poles of $y$ can therefore occur only at the cusps. For a cusp represented by \(a/c\), where \(c\mid 6\) and \(\gcd\left(a,c\right)=1\), By Ligozat's formula \cite{RouseWebb2015, Ligozat1975, McMurdy2001},
\[\operatorname{ord}_{a/c}\left(y\right)=\frac{6}{24}
\sum_{\delta\mid 6}\frac{\gcd\left(c,\delta\right)^2r_\delta}{\gcd\left(c,6/c\right)c\delta}.\]
Substituting the four possible values \(c=1,2,3,6\) gives
\[
\begin{array}{c|c|c}
c & \text{cusp class}
  & \operatorname{ord}_{a/c}\left(y\right)\\
\hline
1 & 0              & 0\\
2 & \frac12       & -1\\
3 & \frac13       & 0\\
6 & \infty         & 1
\end{array}
\]
Thus, \(y\) is meromorphic at every cusp: it has a simple pole at the cusp \(1/2\), a simple zero at \(\infty\), and neither a zero nor a pole at the other two cusps. Consequently, \(y\) is a weight-zero modular function on \(\Gamma_0\left(6\right)\). Since $\Gamma_1\left(6\right)\subseteq\Gamma_0\left(6\right)$, then \(y\) is also a modular function on
\(\Gamma_1\left(6\right)\).
\end{proof}
Now, since $y\left(\tau\right)$ has only one simple zero in the fundamental domain, it generates the field of modular functions on $\Gamma_1\left(6\right)$. Moreover, $y\left(0\right)=\frac{1}{9}$, $y\left(\frac{1}{3}\right)=1$, $y\left(\frac{1}{2}\right)=\infty$, and $y\left(\infty\right)=0$. Also, the function $y\left(-\frac{1}{6\tau}\right)$ is invariant on $\Gamma_1\left(6\right)$ and \begin{align}\label{eqn: eqn:y-fricke-transformation}
y\left(-\frac{1}{6\tau}\right)=\frac{y\left(\tau\right)-1/9}{y\left(\tau\right)-1}.
\end{align}
Hence, \begin{align}\label{eqn: equation involving t} t\left(\tau\right)=y\left(\tau\right)\frac{1-9y\left(\tau\right)}{1-y\left(\tau\right)}
\end{align}
is invariant under the involution $\tau\mapsto -\frac{1}{6\tau}$. Also, let 
\[\Delta\left(\tau\right)=q\prod_{n=1}^{\infty}\left(1-q^n\right)^{24}=\left(\eta\left(\tau\right)\right)^{24}\]
denote the modular discriminant, which is a cusp form of weight 12 for $\operatorname{SL}_2\left(\mathbb{Z}\right)$. Then, \begin{align}\label{eqn: beukers formula for t}
t\left(\tau\right)=\left(\frac{\Delta\left(6\tau\right)\Delta\left(\tau\right)}{\Delta\left(3\tau\right)\Delta\left(2\tau\right)}\right)^{1/2}=\left(\frac{\eta\left(6\tau\right)\eta\left(\tau\right)}{\eta\left(3\tau\right)\eta\left(2\tau\right)}\right)^{12}=q\prod_{n=0}^{\infty}\left(1-q^{6n+1}\right)^{12}\left(1-q^{6n+5}\right)^{12}.
\end{align}
Indeed, the eta quotient
\[\left(\frac{\Delta\left(6\tau\right)\Delta\left(\tau\right)}{\Delta\left(3\tau\right)\Delta\left(2\tau\right)}\right)^{1/2}\]
is a modular function for \(\Gamma_1\left(6\right)\). It is invariant under the Fricke involution $\tau\mapsto -\frac{1}{6\tau}$ and its zeros and poles agree, with the same multiplicities, with those of \(t\left(\tau\right)\). Consequently, the quotient of these two functions is constant. Comparing their leading terms in the \(q\)-expansion shows that this constant is \(1\). At this juncture, we also point out that there is an erratum in Beukers' paper as the exponent of $1-q^{6n+5}$ in (\ref{eqn: beukers formula for t}) should be 12 instead of $-12$ \cite{Beukers1987}.
\begin{proposition}
The function $t\left(\tau\right)$, as defined in (\ref{eqn: beukers formula for t}) maps the shaded open area in Figure \ref{fig:region-I-under-t} univalently onto $\mathcal{H}$ and satisfies the following properties: \[t\left(i\infty\right)=0,\quad t\left(\frac{i}{\sqrt{6}}\right)=\left(\sqrt{2}-1\right)^4,\quad t\left(\frac{2}{5}+\frac{i}{5\sqrt{6}}\right)=\left(\sqrt{2}+1\right)^4,\quad t\left(\frac{1}{2}\right)=\infty.\]
\end{proposition}
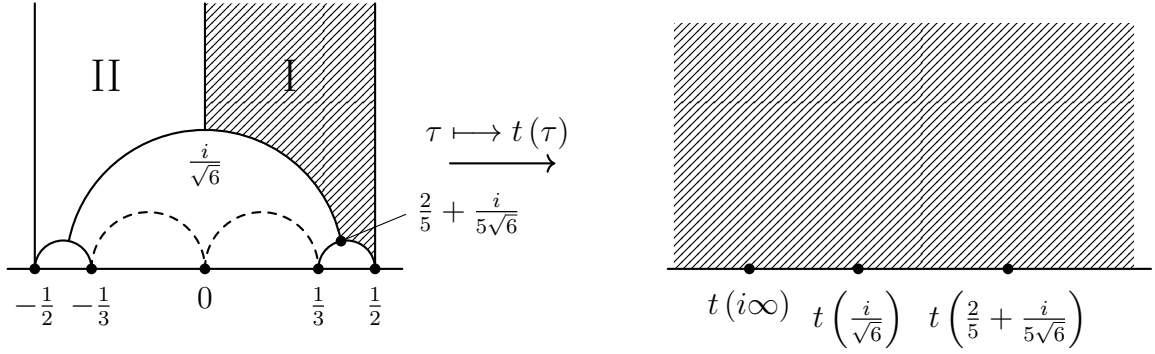
\begin{figure}[H]
    \centering
    \begin{tikzpicture}[
        scale=4.5,
        line cap=round,
        line join=round
    ]

\def\H{0.78}

\pgfmathsetmacro{\Rbig}{1/sqrt(6)}
\pgfmathsetmacro{\rsmall}{1/12}
\pgfmathsetmacro{\thetaP}{asin(1/5)}
\pgfmathsetmacro{\thetaSmall}{acos(-1/5)}

\coordinate (P) at ({2/5},{1/(5*sqrt(6))});
\coordinate (Q) at (0,{1/sqrt(6)});

\path[
    pattern=north east lines,
    pattern color=black
]
    (0,\H)
    -- (1/2,\H)
    -- (1/2,0)
    arc[
        start angle=0,
        end angle=\thetaSmall,
        radius=\rsmall
    ]
    arc[
        start angle=\thetaP,
        end angle=90,
        radius=\Rbig
    ]
    -- cycle;

\draw[thick]
    (-1/2,0)
    --
    (-1/2,\H);

\draw[thick]
    (1/2,0)
    --
    (1/2,\H);

\draw[thick]
    (-0.58,0)
    --
    (0.58,0);


\draw[thick]
    (-1/2,0)
    arc[
        start angle=180,
        end angle=0,
        radius=1/12
    ];

\draw[thick,dashed]
    (-1/3,0)
    arc[
        start angle=180,
        end angle=0,
        radius=1/6
    ];

\draw[thick,dashed]
    (0,0)
    arc[
        start angle=180,
        end angle=0,
        radius=1/6
    ];

\draw[thick]
    (1/3,0)
    arc[
        start angle=180,
        end angle=0,
        radius=1/12
    ];

\draw[thick]
    (P)
    arc[
        start angle=\thetaP,
        end angle={180-\thetaP},
        radius=\Rbig
    ];

\draw[thick]
    (Q)
    --
    (0,\H);

\node at (0.25,0.57)
    {\Large $\mathrm{I}$};

\node at (-0.29,0.57)
    {\Large $\mathrm{II}$};

\node[
    right=2pt,
    below=1pt
] at (Q)
    {$\frac{i}{\sqrt6}$};

\fill
    (P)
    circle (0.45pt);

\draw
    (P)
    --
    (0.59,0.16);

\node[right] at (0.59,0.16)
    {$\frac25+\frac{i}{5\sqrt6}$};

\foreach \x/\lab in {
    -0.5/{-\frac12},
    -0.333333/{-\frac13},
     0/{0},
     0.333333/{\frac13},
     0.5/{\frac12}
}{
    \fill
        (\x,0)
        circle (0.45pt);

    \node[below=3pt] at (\x,0)
        {$\lab$};
}

\node at (0.86,0.40)
    {$\tau\longmapsto t\left(\tau\right)$};

\draw[->,thick]
    (0.72,0.31)
    --
    (1.03,0.31);

\begin{scope}[shift={(1.38,0)}]

    \path[
        pattern=north east lines,
        pattern color=black
    ]
        (0,0)
        rectangle
        (1.35,0.72);

    \draw[thick]
        (-0.02,0)
        --
        (1.40,0);

    \coordinate (Tinf) at (0.22,0);
    \coordinate (Ti)   at (0.54,0);
    \coordinate (Tp)   at (0.98,0);

    \fill (Tinf) circle (0.45pt);
    \fill (Ti)   circle (0.45pt);
    \fill (Tp)   circle (0.45pt);

    \node[below=4pt] at (Tinf)
        {$t\left(i\infty\right)$};

    \node[below=4pt] at (Ti)
        {$t\left(\frac{i}{\sqrt6}\right)$};

    \node[below=4pt] at (Tp)
        {$t\left(
            \frac25+\frac{i}{5\sqrt6}
        \right)$};

\end{scope}

    \end{tikzpicture}

    \caption{The region $\mathrm{I}$ and its image under the map $\tau\mapsto t\left(\tau\right)$}
    \label{fig:region-I-under-t}
\end{figure}
\begin{proof}
Recall the auxiliary modular function (\ref{eqn: definition of y}) and that Proposition \ref{prop: y is modular gamma 1 6} asserts that \(y\) is a modular function on \(\Gamma_1\left(6\right)\). The values at the four inequivalent cusps are $y\left(i\infty\right)=0$, $y\left(0\right)=\frac{1}{9}$, $y\left(\frac{1}{3}\right)=1$, and $y\left(\frac{1}{2}\right)=\infty$. 

Moreover, \(y\) has a unique simple pole on the compact modular curve \(X_1\left(6\right)\). It therefore defines a degree-one map $X_1\left(6\right)\rightarrow\mathbb{P}^1\left(\mathbb{C}\right)$ and hence \(y\) is a Hauptmodul for \(\Gamma_1\left(6\right)\).

Let $W_6\tau=-\frac{1}{6\tau}$. A direct application of the eta transformation formula gives
\begin{align}\label{eqn: y fricke 2}
y\left(W_6\tau\right)=\frac{y\left(\tau\right)-1/9}{y\left(\tau\right)-1}
\end{align}
as in (\ref{eqn: eqn:y-fricke-transformation}). Define
\[R\left(Y\right)=\frac{Y\left(1-9Y\right)}{1-Y}.\]
Then, $t\left(\tau\right)=R\left(y\left(\tau\right)\right)$. If
\[\phi\left(Y\right)=\frac{Y-1/9}{Y-1},
\]
then
\[R\left(\phi\left(Y\right)\right)=R\left(Y\right).\]
Thus,
\[t\left(W_6\tau\right)=t\left(\tau\right).\]
More precisely, for \(X,Y\in\mathbb{P}^1\left(\mathbb{C}\right)\),
\[R\left(X\right)-R\left(Y\right)=\frac{\left(X-Y\right)\left(9XY-9X-9Y+1\right)}{\left(X-1\right)\left(Y-1\right)}.\]
Consequently,
\[R\left(X\right)=R\left(Y\right)\quad\text{if and only if}\quad X=Y\text{ or }X=\phi\left(Y\right).
\]
Since \(y\) is a Hauptmodul, it follows that two points have the same \(t\)-value precisely when they are equivalent either under \(\Gamma_1\left(6\right)\) or under \(W_6\Gamma_1\left(6\right)\). Hence \(t\) is a Hauptmodul for the extended group $\Gamma_1\left(6\right)^{+}=\left\langle\Gamma_1\left(6\right),W_6\right\rangle$, where $\left\langle A,B\right\rangle$ denotes the group generated by $A$ and $B$. 

The regions labelled \(\mathrm{I}\) and \(\mathrm{II}\) in Figure \ref{fig:region-I-under-t} together form a fundamental region for this extended group. Therefore, \(t\) assumes every value at most once in \(\mathrm{I}\cup\mathrm{II}\), and in particular it is injective on the shaded open region \(\mathrm{I}\). So, $t$ is univalent.

We next show that the boundary of \(\mathrm{I}\) is mapped into \(\mathbb{R}\cup\left\{\infty\right\}\). Indeed, each boundary arc is paired with its complex-conjugate arc by an element of \(\Gamma_1\left(6\right)^{+}\). Since the Fourier expansion of \(t\) has real
coefficients,
\[t\left(-\overline{\tau}\right)=\overline{t\left(\tau\right)}.\]
For a point \(\tau\) on a boundary arc, the side-pairing therefore gives
\[t\left(\tau\right)=t\left(-\overline{\tau}\right)=\overline{t\left(\tau\right)},\]
so that \(t\left(\tau\right)\) is real. It follows that \(t\) maps the shaded region conformally onto one of the two
half-planes. To determine which one, take $\tau=\frac{1}{4}+iT$ with \(T\) sufficiently large. Since
\[q=e^{2\pi i\tau}=ie^{-2\pi T}\quad\text{and}\quad t\left(\tau\right)=q+O\left(q^2\right),\]
then $\operatorname{Im}\left(t\left(\tau\right)\right>0$ for sufficiently large \(T\). Thus the image is the upper half-plane \(\mathcal{H}\). Hence, \(t\) maps the shaded open region univalently onto \(\mathcal{H}\). It remains to calculate the distinguished boundary values. First,
\[t\left(i\infty\right)=R\left(y\left(i\infty\right)\right)=R\left(0\right)=0.\]
Similarly, since \(y\left(\frac{1}{2}\right)=\infty\) and $R\left(Y\right)\sim 9Y$ as $Y\to\infty$, then $t\left(\frac{1}{2}\right)=\infty$. 

Now, put $\tau_0=\frac{i}{\sqrt{6}}$. This point is fixed by \(W_6\) since $-\frac{1}{6\tau_0}=\tau_0$. Let $y_0=y\left(\tau_0\right)$. Then, (\ref{eqn: y fricke 2}) gives
\[y_0=\frac{y_0-1/9}{y_0-1}.\]
Therefore,
\[y_0^2-2y_0+\frac{1}{9}=0\quad\text{so}\quad y_0=1\pm\frac{2\sqrt{2}}{3}.
\]
Substituting \(t=R\left(y\right)\) gives the two possible values
\[t\left(\tau_0\right)=17\pm12\sqrt{2}=\left(\sqrt{2}\pm1\right)^4.\]
At \(\tau_0\), the nome is $q_0=e^{-\frac{2\pi}{\sqrt{6}}}\in\left(0,1\right)$. Using the infinite-product expression
\[t\left(\tau_0\right)=q_0\prod_{n=0}^{\infty}\left(1-q_0^{6n+1}\right)^{12}\left(1-q_0^{6n+5}\right)^{12},\]
we see that $0<t\left(\tau_0\right)<1$. We must therefore take the smaller of the two possible values:
\[t\left(\frac{i}{\sqrt{6}}\right)=17-12\sqrt{2}
=\left(\sqrt{2}-1\right)^4.\]
Finally, consider the Atkin-Lehner involution $W_2$ at level 6 
\[W_2\tau=\frac{2\tau-1}{6\tau-2}
\]
as in \cite{AtkinLehner1970}. The eta transformation law, or equivalently the action of \(W_2\) on the four cusps, gives
\[
y\left(W_2\tau\right)
=
\frac{1-y\left(\tau\right)}
{1-9y\left(\tau\right)}.
\]
Since
\[
R\left(\frac{1-Y}{1-9Y}\right)
=
\frac{1}{R\left(Y\right)},
\]
we obtain
\begin{align}\label{eqn: w2 transformation}
t\left(W_2\tau\right)=\frac{1}{t\left(\tau\right)}.
\end{align}
Furthermore,
\[W_2\left(\frac{i}{\sqrt{6}}\right)=\frac{2i/\sqrt{6}-1}{6i/\sqrt{6}-2}=\frac{2}{5}+\frac{i}{5\sqrt{6}}.
\]
It follows from (\ref{eqn: w2 transformation}) that
\begin{align*}
t\left(\frac{2}{5}+\frac{i}{5\sqrt{6}}\right)=\frac{1}{t\left(\frac{i}{\sqrt{6}}\right)}=\frac{1}{\left(\sqrt{2}-1\right)^4}=\left(\sqrt{2}+1\right)^4.
\end{align*}
Therefore, we obtain the mentioned values of $t$ at the two cusps and the two distinguished interior points.
\end{proof}

In the theorems and proofs that we encounter in due course, let $M_k\left(\Gamma_1\left(6\right)\right)$ denote the space of modular forms of weight $k$ with respect to $\Gamma_1\left(6\right)$, and let \[E_4\left(\tau\right)=1+240\sum_{n=1}^{\infty}\sigma_3\left(n\right)q^n\quad\text{and}\quad E_2\left(\tau\right)=1-24\sum_{n=1}^{\infty}\sigma\left(n\right)q^n\]
be Eisenstein series. Here, \[\sigma_z\left(n\right)=\sum_{d\mid n}d^z\]
denotes the sum of positive divisors function.
\begin{theorem}[Apéry]
$\zeta\left(3\right)$ is irrational.
\end{theorem}
\begin{proof}
Define the functions $F$ and $E$ satisfying \begin{align*}
    40F\left(\tau\right)&=E_4\left(\tau\right)-36E_4\left(6\tau\right)-28E_4\left(2\tau\right)+63E_4\left(3\tau\right)\\
    24E\left(\tau\right)&=-5E_2\left(\tau\right)+30E_2\left(6\tau\right)+2E_2\left(2\tau\right)-3E_2\left(3\tau\right)
\end{align*}
The function \(F\) belongs to \(S_4\left(\Gamma_1\left(6\right)\right)\), the complex vector space of cusp forms of weight 4 for the congruence subgroup $\Gamma_1\left(6\right)$, and satisfies the Fricke transformation law
\[F\left(-\frac{1}{6\tau}\right)=-36\tau^4F\left(\tau\right).\]
The Dirichlet series corresponding to $F\left(\tau\right)$ is \begin{align*}
L\left(F,s\right)&=\sum_{n=1}^{\infty}\left[\frac{6\sigma_3\left(n\right)}{n^s}-36\frac{6\sigma_3\left(n\right)}{\left(6n\right)^s}-28\frac{6\sigma_3\left(n\right)}{\left(2n\right)^s}+63\frac{6\sigma_3\left(n\right)}{\left(3n\right)^s}\right]\\
&=6\left(1-6^{2-s}-7\cdot 2^{2-s}+7\cdot 3^{2-s}\right)\zeta\left(s\right)\zeta\left(s-3\right)
\end{align*}
where we used the fact that \[\sum_{n=1}^{\infty}\frac{\sigma_k\left(n\right)}{n^s}=\zeta\left(s\right)\zeta\left(s-k\right)\quad\text{where }\operatorname{Re}\left(s\right)>k+1.\]
One can use the functional equation for the Riemann zeta function \[\zeta\left(s\right)=2^s\pi^{s-1}\sin \left(\frac{\pi s}{2}\right)\Gamma\left(1-s\right)\zeta\left(1-s\right)\]
and the gamma function $\Gamma\left(s+1\right)=s\Gamma\left(s\right)$ to show that $\zeta\left(0\right)=-\frac{1}{2}$. Consequently, \[L\left(F,3\right)=\zeta\left(3\right).\]
Let \(f\) be the Fourier series normalised by
\[f'''\left(\tau\right)=\left(2\pi i\right)^3F\left(\tau\right)\quad\text{where }
f\left(i\infty\right)=0.
\]
It follows from Proposition \ref{prop: h formula in terms of d} that
\begin{align}\label{eqn: transformation of f}
6\tau^2\left(f\left(-\frac{1}{6\tau}\right)-L\left(F,3\right)\right)=-f\left(\tau\right)+L\left(F,3\right)-2\pi i\tau L\left(F,2\right).
\end{align}
One can show that $L\left(F,2\right)=0$, so \[6\tau^2\left(f\left(-\frac{1}{6\tau}\right)-\zeta\left(3\right)\right)=\zeta\left(3\right)-f\left(\tau\right).\]
Set \[H\left(\tau\right)=E\left(\tau\right)\left(f\left(\tau\right)-\zeta\left(3\right)\right).\]
One can show that $E$ satisfies the Fricke transformation law \[E\left(-\frac{1}{6\tau}\right)=-6\tau^2E\left(\tau\right)\quad\text{so}\quad H\left(-\frac{1}{6\tau}\right)=H\left(\tau\right).\]
Recall our formula for $t$ in (\ref{eqn: beukers formula for t}) obtained from the product expansion of $\eta$. Then, $t=q+O\left(q^2\right)$ so $t$ has a local compositional inverse $q=q\left(t\right)\in t\mathbb{Z}\left[\left[t\right]\right]$ near $t=0$. As such, we regard $E,f,H$ as power series in $t$. 

We previously established that the first two positive branching values are $t_1=(\sqrt{2}-1)^4$ and $t_2=(\sqrt{2}+1)^4$. The first branching value corresponds to the fixed point $\tau=\frac{i}{\sqrt{6}}$ of the Fricke involution. The two local branches of the inverse $\tau=\tau\left(t\right)$ near $t=t_1$ are interchanged by $\tau \mapsto -\frac{1}{6\tau}$. Note that the function $H$ takes the same value on these two branches. Hence, the apparent branch singularity at $t=t_1$ is removable for $H$. The next branching value is $t_2$, and therefore the Taylor series of $H\left(t\right)$ about $t=0$ has radius of convergence $\rho=t_2=17+12\sqrt{2}$. In particular, this radius is finite, so $H\left(t\right)$ is not a polynomial and has infinitely many non-zero Taylor coefficients.

It remains to control the denominators. Note that the Fourier coefficients $a_n$ of $F$ are integers. So, consider the Eichler integral of $F$ by \[f\left(\tau\right)=\sum_{n=1}^{\infty}\frac{a_n}{n^3}q^n\]
which has rational coefficients whose $n^\text{th}$ denominator divides $n^3$. Let $D_n=\operatorname{lcm}\left(1,2,\ldots,n\right)$. SInce $q\left(t\right)\in t\mathbb{Z}\left[\left[t\right]\right]$, the coefficient of $t^n$ in $f\left(q\left(t\right)\right)$ is an integral linear combination of $\frac{a_m}{m^3}$, where $1\le m\le n$. Since $m\mid D_n$ for all $m\le n$, then its denominator divides $D_n^3$. 

Likewise, $E\left(q\right)\in \mathbb{Z}\left[\left[q\right]\right]$ so $E\left(t\right)\in\mathbb{Z}\left[\left[t\right]\right]$. Write \[E\left(t\right)f\left(t\right)=\sum_{n=0}^{\infty}A_nt^n\quad\text{and}\quad E\left(t\right)=\sum_{n=0}^{\infty}B_nt^n.\]
So, $A_n\in \frac{1}{D_n^3}\mathbb{Z}$ and $B_n\in\mathbb{Z}$ and \[H\left(t\right)=\sum_{n=0}^{\infty}\left(A_n-\zeta\left(3\right)B_n\right)t^n.\]
Apply Proposition \ref{prop: Beukers' irrationality criterion} with $f_0\left(t\right)=E\left(t\right)f\left(t\right)$, $f_1\left(t\right)=-E\left(t\right)$, and $\theta=\zeta\left(3\right)$. Also, set $d=1$ and $r=3$. Then, $H\left(t\right)$ has infinitely many non-zero coefficients and $\rho=17+12\sqrt{2}>e^3$. This implies that $\zeta\left(3\right)$ is irrational.
\end{proof}

\end{document}